\documentclass{amsart}
\usepackage{amsfonts,amssymb,amscd,amsmath,enumerate,verbatim,calc,graphicx}
\usepackage[all]{xy}
\newtheorem{theorem}{Theorem}[section]
\newtheorem{corollary}[theorem]{Corollary}
\newtheorem{lemma}[theorem]{Lemma}
\newtheorem{proposition}[theorem]{Proposition}
\theoremstyle{definition}
\newtheorem{definition}[theorem]{Definition}

\newtheorem{example}[theorem]{Example}

\newcommand{\N}{\mathbb{N}}

\newcommand{\Z}{\mathbb{Z}}

\newcommand{\sub}{\subseteq}

\newcommand{\ov}{\overline}

\newcommand{\lo}{\longrightarrow}
\newcommand{\D}{\displaystyle}

\newcommand{\wt}{\widetilde}
\newcommand{\vf}{\varphi}

\newcommand{\al}{\alpha}
\newcommand{\bt}{\beta}
\newcommand{\la}{\lambda}
\newcommand{\ka}{\kappa}
\newcommand{\peb}{p:(E,\kappa)\rightarrow (B,\lambda)}

\newcommand{\kl}{(\kappa, \lambda)}
\newcommand{\stk}{\stackrel{\ka}{\leftrightarrow}}
\newcommand{\stl}{\stackrel{\la}{\leftrightarrow}}
\begin{document}

\author[A. Pakdaman , M. Zakki]
{  Ali~Pakdaman$^{1,*}$, Mehdi~Zakki$^{1}$}

\title[Equivalent Conditions for Digital Covering Maps]
{Equivalent Conditions for Digital Covering Maps}
\thanks{$^*$Corresponding author}
\thanks{E-mail addresses: a.pakdaman@gu.ac.ir, mehdizakki14@gmail.com}
\maketitle

\begin{center}
{\it $^1$Department of Mathematics, Faculty of Sciences, Golestan University,\\
P.O.Box 155, Gorgan, Iran.}
\end{center}

\vspace{0.4cm}
\begin{abstract}
It is known that every digital covering map $p:(E,\kappa)\rightarrow (B,\lambda)$ has unique path lifting property. In this paper, we show that its inverse is true when the continuous surjection map $p$ has no conciliator point. Also, we prove that a digital $(\kappa,\lambda)-$continuous surjection $p:(E,\kappa)\rightarrow (B,\lambda)$ is a digital covering map if and only if it is a local isomorphism.  Moreover, we find out a loop criterion for a digital covering map to be a radius $n$ covering map.
\end{abstract}
\vspace{0.5cm}
%\\\\\\\\\\\\\\\\\\\\\\\\\\\\\\\\\\\\\\\\\\\\\\\\\\\\\\\\\\\\\\\\\\\\\\\\\\\\\\\\\\\\\\\\\\\\\\\\\\\\\\\\\\\\\\\\\\\\\\\\\\\\\\\\\\\\\\\\\
%=========================================================================================================================================
%/////////////////////////////////////////////////////////////////////////////////////////////////////////////////////////////////////////
\section{Introduction and Motivation}

In image processing, computer graphics and modeling topology in medical
image processing algorithms, an object in the plane or $3$-space is often
approximated digitally by a set of pixels or voxels. Digital topology deals with topological properties
of this set of pixels or voxels that correspond to topological properties of the original
object. It provides theoretical foundations for important operations such as digitization,
connected component labeling and counting, boundary extraction, contour filling, and
thinning. Digitization
is replacing an object by a discrete set of its points\cite{KR,Kov}.

In recent years, computing topological invariants has been of great importance in understanding the shape of an arbitrary $2$-dimensional
(2D) or $3$-dimensional (3D) object \cite{Kac}. The most powerful invariant of these objects is the fundamental group \cite{Sp}, which is unfortunately
difficult to work with, although for 3D objects, this problem is decidable but no practical algorithm
has been found yet.

The digital fundamental group of a discrete object was introduced in Digital Topology
by Kong and Stout \cite{K,ST}. Boxer \cite{B2} has shown how classical methods of Algebraic Topology may be used
to construct the digital fundamental group which is useful tool for Image Analysis. The digital covering space is an important tool for computing fundamental groups of digital images. A digital covering space has been introduced by Han \cite{H1}. Boxer \cite{B3} has developed further the topic of digital covering space by deriving digital analogs of classical results
of Algebraic Topology concerning the existence and properties of digital universal covering
spaces. Boxer and Karaca \cite{B4, B5} have classified digital covering spaces by the conjugacy
classes of image subgroups corresponding to a digital covering space.

   Lots of the researches in digital covering theory are digitization of the concepts in Topology and Algebraic Topology. In Algebraic Topology, covering maps are local isomorphism (local homeomorphism) and also satisfy all the lifting problems. But the converse is not true. In fact, every local isomorphism is not necessarily a covering map and there are some maps in which enjoy various concepts of lifting, but are not a covering map.

   Despite the behavior of spaces in General Topology are locally complicated, digital images are locally simple and this leads us to investigate the conditions such that makes some concepts to be equivalent. For this, after some reminders and preliminary results about digital topology and digital covering map, we introduce notions ``digital path lifting property", ``uniqueness of digital path lifts" and ``unique digital path lifting property"
for a digitally continuous map and will compare them by some examples.

 Digital path lifting property means that every digital path has a lifting started at a given point in the
appropriate fiber. By uniqueness of digital path lifts we mean that if a digital path has a lifting at a given point, it must be unique. Eventually, a map has unique digital path lifting property if it has both of digital path lifting property and uniqueness of digital path lifts.
Every digital covering map has unique digital path lifting property \cite{H1}.
By proving some basic results about maps equipped with such properties, we show that every continuous surjection with unique digital path lifting property is a covering map when it has no conciliator point and by some example will emphasis that these hypotheses are essential. This shows that digital covering theory is not a special case of the well known concept of a graph covering projection because in covering graph theory a graph map is a covering graph if and only if it has unique path lifting property \cite{Boldi}.

In General Topology, a covering map is a local isomorphism and local isomorphisms are not necessarily covering map. Inspired by this, in the all of the researches it is claimed by a misplaced example that digital local isomorphisms are not necessarily covering map (for example see\cite{B3,H3}).
 We show that every digital local isomorphism is a covering map. Digital versions of some fundamental theorems in Algebraic Topology are satisfied for radius 2 local isomorphisms and this motivates us to find out a loop criterion for a digital covering map to be a local isomorphism with radius $n$.

\section{Notations and preliminaries}
Let $\mathbb{Z}$ be the set of integers. Then $\Z^n$ is the set of lattice points in the $n$-dimensional
Euclidean space. Let $X\sub\Z^n$ and let $\ka$ be some adjacency relation for the members of
$X$. Then the pair $(X, \ka)$ is said to be a (binary) digital image. For a positive integer $u$ with $1\leq u \leq n$,
an adjacency relation of a digital image in $\Z^n$ is defined as follows:\\
Two distinct points
$p = (p_1, p_2, . . . , p_n)$ and $q = (q_1, q_2, . . . , q_n)$ in $\Z^n$ are $l_u$-adjacent \cite{Ros1} if
there are at most $u$ distinct indices $i$ such that $|p_i - q_i| = 1$ and for all indices $j$, $p_j=q_j$ if $|p_j - q_j | \neq 1$.
An $l_u$-adjacency relation on $\Z^n$ can be denoted by the number of points that are $l_u$-adjacent
to a given point $p\in\Z^n$. For example,
\begin{itemize}
\item The $l_1$-adjacent points of $\Z$ are called 2-adjacent.
\item The $l_1$-adjacent points of $\Z^2$ are called 4-adjacent and the $l_2$-adjacent points in $\Z_2$ are
called 8-adjacent.
\item The $l_1$-adjacent, $l_2$-adjacent and $l_3$-adjacent points of $\Z^3$ are called 6-adjacent, 18-adjacent, and 26-adjacent, respectively.
\end{itemize}
More general adjacency relations are studied in \cite{Her}.

Let $\ka$ be an adjacency relation defined on $\Z^n$. A digital image $X\sub \Z^n$ is \textbf{$\ka$-connected}
 \cite{Ros2} if and only if for every pair of different points $x, y \in X$, there is a set $\{x_0, x_1, . . . , x_r\}$ of points of $X$ such that $x = x_0, y = x_r$ and $x_i$ and $x_{i+1}$ are $\ka$-adjacent
where $i = 0, 1, . . . , r - 1$. A $\ka$-component of a digital image $X$ is a maximal $\ka$-connected
subset of $X$.

\begin{definition}
Let $X \sub \Z^n$ and $Y \sub \Z^m$ be digital images with $\ka$-adjacency and $\la$-adjacency, respectively.
A function $f : X \lo Y$ is said to be \textbf{$(\ka, \la)$-continuous} (\cite{B2,Ros2}) if for every
$\ka$-connected subset $U$ of $X$, $f(U)$ is a $\la$-connected subset of $Y$. We say that such a
function is digitally continuous.
\end{definition}
The following proposition let us to interpret digital continuity by adjacency relations.
\begin{proposition}(\cite{B2,Ros1})
Let $(X,\ka)$ in $\Z^{n}$ and $(Y,\la)$ in $\Z^m$ be digital images.
A function $f : X \lo Y$ is {$(\ka, \la)$-continuous} if and
only if for every $\ka$-adjacent points $x_0, x_1\in X$, either $f(x_0) = f(x_1)$ or $f(x_0)$ and $f(x_1)$
are $\la$-adjacent in $Y$.
\end{proposition}
For $a,b\in\Z$ with $a<b$, a \textbf{digital interval} \cite{B1} is the set of the form
$$[a, b]_{\Z} = \{z \in \Z|a \leq z \leq b\}.$$
\begin{definition}
By a \textbf{digital $\ka$-path} from $x$ to $y$ in digital image $(X,\ka)$, we mean a $(2, \ka)$-continuous function
$f : [0,m]_{\Z} \lo X$ such that $f(0) = x$ and $f(m) = y$. If $f(0) = f(m)$ then the $\ka$-path is
said to be closed, and f is called a $\ka$-loop.
\end{definition}
 Let $f : [0,m - 1]_{\Z} \lo X\sub \Z^n$ be a
$(2, \ka)$-continuous function such that $f(i)$ and $f(j)$ are $\ka$-adjacent if and only if $j = i \pm 1$
mod $m$. Then the set $f([0,m - 1]_{\Z})$ is a simple closed $\ka$-curve containing $m$ point which is denoted by $SC_{\ka}^{n,m}$. If $f$ is a constant function, it is called a trivial loop.

If $f : [0,m_1]_{\Z} \lo X$ and $g : [0,m_2]_{\Z} \lo X $ are digital $\ka$-paths with $f(m_1) = g(0)$, then
define the product \cite{KH} $(f * g) : [0,m_1 + m_2]_{\Z} \lo X$ by
$$
(f * g)(t) =\begin{cases}
f(t) \ \ if \ \ t \in [0,m_1]_{\Z};\\
g(t - m_1)\ \  if\ \  t \in [m_1,m_1 + m_2]_{\Z}.
\end{cases}$$

Let $(E, \ka)$ be a digital image and let $\varepsilon\in N$. The $\ka$-neighborhood [8] of $e_0\in E$ with
radius $\varepsilon$ is the set
$N(e_0, \varepsilon) = \{e \in E |~ l_{\ka}(e_0, e) \leq \varepsilon\} \cup \{e_0\}$,
where $l_{\ka}(e_0, e)$ is the length of the shortest $\ka$-path in $E$ from $e_0$ to $e$.

By the above notations, a function $f : X \lo Y$ is a \textbf{$(\ka,\la)$-isomorphism} \cite{B3}, denoted by $X\stackrel{(\ka,\la)}{\approx}Y$, if $f$ is a $(\ka,\la)$-continuous bijection and further $f^{-1} : Y \lo X$ is $(\la,\ka)$-continuous. If $n = m$ and $\ka=\la$, then $f$ is called $\ka$-isomorphism.

\begin{definition}(\cite{B2})
Let $(X,\ka)$ and $(Y,\la)$ be digital images and let $f, g : X\lo Y$ be $(\ka, \la )$-continuous functions. Suppose that there is a
positive integer $m$ and a function $F : X \times [0,m]_{\Z}\lo Y$ such that
\begin{itemize}
\item For all $x\in X$, $F(x, 0) = f(x)$ and $F(x,m) = g(x)$;
\item For all $x\in X$, the induced function $F_x : [0,m]_{\Z}\lo Y$ defined by $F_x(t) = F(x, t)$ for all
$t\in [0,m]_{\Z}$ is $(2, \la )$-continuous; and
\item For all $t \in [0,m]_{\Z}$, the induced function $F_t : X\lo Y$ defined by $F_t(x) = F(x, t)$ for all
$x \in X$ is $(\ka, \la )$-continuous.
\end{itemize}
Then $F$ is called a \textbf{digital $(\ka, \la )$-homotopy} between $f$ and $g$, denoted by $f \stackrel{(\ka,\la)}{\simeq} g$, and
$f$ and $g$ are said to be digital $(\ka, \la )$-homotopic in $Y$ by $F$.
\end{definition}
Digital $(\ka,\la)$-homotopy relation is an equivalence relation
among digitally continuous functions $f : (X, \ka) \lo (Y, \la)$ \cite{B2}.

Let $f$ and $f'$ be $\ka$-loops in the pointed digital image $(X, x_0)$. We say $f'$ is a trivial extension
of $f$ if there are sets of $\ka$-paths $\{f_1, f_2, . . . , f_r\}$ and $\{F_1, F_2, . . . , F_p\}$ in $X$ such that\\
(1) $r \leq p$;\\
(2) $f = f_1*f_2 * ... * f_r$;\\
(3) $f_0 = F_1 * F_2 * ... * F_p$;\\
(4) There are indices $1 \leq i_1 < i_2 < ... < i_r \leq p$ such that $F_{i_j} = f_j,\  1 \leq j \leq r$ and
$i \notin \{i_1, i_2, . . . , i_r\}$ implies $F_i$ is a trivial loop\cite{B2}.

Two loops $f$, $f'$ with the same base point $x_0 \in X$ belong to the same loop class $[f]_X$
if they have trivial extensions that can be joined by a homotopy that keeps the endpoints
fixed. Let $\pi_1^{\ka}(X,x_0)$ be the set of all such classes, $[f]_X$.
The operation $*$ enables us to define an operation on $\pi_1^{\ka}(X,x_0)$ via
$$[f]_X .[g]_X = [f * g]_X.$$
This operation is well defined, and makes $\pi_1^{\ka}(X,x_0)$ into a group in which the
identity element is the class $[\overline{x_0}]$ of the constant loop $\ov{x_0}$ and in which inverse elements
are given by $[f]^{-1} = [f^{-1}]$, where $f^{-1} :[0,m]_{\Z}\lo X$ is the loop defined
by
$f^{-1}(t)=f(m-t)$\cite{B2}.

\begin{definition}\cite{H2}
For two digital spaces $(X, \ka)$ in $\Z^{n}$ and $(Y,\la)$ in $\Z^m$, a $(\ka,\la)$-continuous
map $h : X \lo Y$ is called \textbf{local $(\ka,\la)$-isomorphism} if for every $x\in X$, $h|_{N_{\ka} (x; 1)}$ is a $(\ka,\la)$-isomorphism
onto $N_{\la} (h(x); 1)$. If $n = m$ and $\ka=\la$, then the map $h$ is called local $\ka$-isomorphism.
\end{definition}
For $n\in\N$, the map $h$ is called a radius $n$ local isomorphism
if the restriction map $h|_{N_{\ka} (x,n)}:N_{\ka} (x,n)\lo N_{\la}(h(x),n)$ is a $(\ka,\la)$-isomorphism.
\begin{definition}\cite{B3,H1,H2}\label{d1}
Let $(E, \ka)$ and $(B, \la)$ be digital images and $p : E \lo B$ be a
$(\ka,\la)$-continuous surjection map. The map $p$ is called a \textbf{$(\ka,\la)$-covering map} if and only if for
each $b \in B$ there exists an index set $M$ such that\\
(1) $p^{-1}(N_{\la}(b, 1)) = \D\bigsqcup_{i\in M} N_{\ka}(e_i, 1)$ with $e_i \in p^{-1}(b)$;\\
(2) if $i, j \in M$, $i \neq j$, then $N_{\ka} (e_i, 1) \cap N_{\ka} (e_j , 1) = \emptyset$; and\\
(3) the restriction map $p|_{N_{\ka} (e_i,1)} : N_{\ka} (e_i, 1) \lo N_{\la} (b, 1)$ is a $(\ka, \la)$-isomorphism
for all $i \in M$.
\end{definition}
 Moreover, $(E; p; B)$ is said to be a $(\ka,\la)$-covering and $(E,\ka)$ is called
a digital $(\ka,\la)$-covering space over $(B,\la)$. Also, $N_{\la}(b, 1)$ is called an elementary $\la$-neighborhood of $b$.

 It is notable that in the property (1) of the original definition of digital covering map by Han \cite{H1}, there was $N_{\la}(b, \varepsilon)$, for an $n\in\N$ which is simplified by Boxer \cite{B3}. Also, we can replace $(\ka,\la)$-continuous surjection
with surjection because surjective map $p$ with the properties (1) and (3) of definition is $(\ka,\la)$-continuous.

In this paper, all the digital spaces assumed to be connected.

%\\\\\\\\\\\\\\\\\\\\\\\\\\\\\\\\\\\\\\\\\\\\\\\\\\\\\\\\\\\\\\\\\\\\\\\\\\\\\\\\\\\\\\\\\\\\\\\\\\\\\\\\\\\\\\\\\\\\\\\\\\\\\\\\\\\\\\\\\\
%==========================================================================================================================================
%//////////////////////////////////////////////////////////////////////////////////////////////////////////////////////////////////////////
\section{Coverings are derived from unique path lifting }
Like in the Algebraic Topology, digital covering maps have also good behavior with lifting problems. In this section, at first we list some results of the other papers about digital coverings and lifting problems which are digitization of similar results in Algebraic Topology. Then by modification of the notions digital path lifting and unique path lifting, we show how digital covering maps can be derived from unique path lifting property.
\begin{definition}\cite{H1}
Let $(E,\ka)$, $(B,\la)$ and $(X,\mu)$ be digital images, let $p : E \lo B$ be a $(\ka,\la)$-covering
map, and let $f : X \lo B$ be $(\mu,\la)$-continuous. A \textbf{lifting} of $f$ with respect to $p$ is a
$(\mu,\ka)$-continuous map $\wt{f} : X \lo E$ such that $p \circ \wt{f} = f$.
\end{definition}
\begin{theorem}\cite{H1}\label{t01}
Let $(E, \ka)$ be a digital image and $e_0 \in E$. Let $(B, \la)$ be a digital image
and $b_0 \in B$. Let $p : E \lo B$ be a $(\ka,\la)$-covering map such that $p(e_0) = b_0$. Then any
$\la$-path $\al:[0,m]_{\Z} \lo B$ beginning at $b_0$ has a unique lifting to a path $\wt{\al}$ in $E$ beginning at $e_0$.\\
%(2) For any two $\ka$-paths $\al,\bt : [0,m]_{\Z} \lo E$ beginning at $e_0$,
%and satisfying $p \circ \al = p \circ\bt$, we have $\al=\bt$.
\end{theorem}

\begin{definition}
Let $\peb$ be a $\kl$-continuous surjection map. We say that\\
(i) $p$ has \textbf{digital path lifting property} if for any digital path $\al$ in $B$ and any $e\in p^{-1}(\al(0))$ there is a lifting $\wt{\al}$ of $\al$ in $E$  such that $\wt{\al}(0)=e$.\\
(ii) $p$ has the \textbf{uniqueness of digital path lifts property} if any two paths
$\al,\beta:[0,m]_{\Z}\lo E$ are equal if $p\circ\al=p\circ\beta$ and $\al(0)=\beta(0)$.\\
(iii) $p$ has the \textbf{unique path lifting property} (u.p.l, for abbreviation) if it has
both the path lifting property and the uniqueness of path lifts property.
\end{definition}
\begin{example}\label{e1}
By Theorem \ref{t01}, every digital covering map has u.p.l. Consider $\Z^2$ by 8-adjacency and $\Z$ by 2-adjacency. Then the $(8,2)$-continuous map $pr_1:\Z^2\lo \Z$ defined by $pr_1((x,y))=x$ has digital path lifting property, but not uniqueness of digital path lifts property. For this, consider $\al:[0,2]_{\Z}\lo \Z$ defined by $\al(k)=k$, for $k=0,1,2$. $pr_1^{-1}(0)=\{0\}\times \Z$ and for every $(0,j)\in pr_1^{-1}(0)$,  $\wt{\al}_j:[0,2]_{\Z}\lo\Z^2$ defined by $\wt{\al}_j(k)=(k,j)$ is a $8$-path. Then $\wt{\al}_j$ is a lifting of $\al$ beginning at $(0,j)$ and hence $pr_1$ has digital path lifting property. Now let
$$
\begin{cases}
\bt,\gamma:[0,2]_{\Z}\lo\Z^2\\
\bt(0)=(0,0),\ \bt(1)=(1,1),\ \bt(2)=(2,0)\\
 \gamma(0)=(0,0),\ \gamma(1)=(1,-1),\ \gamma(2)=(2,0).
 \end{cases}
$$
Then $pr_1\circ\bt=pr_1\circ\gamma$ and $\bt(0)=\gamma(0)$, but $\bt\neq\gamma$. Therefore $pr_1$ has not uniqueness of digital path lifts property. Also, $pr_1$ is not a digital $(8,2)$-covering because for every $e\in\Z^2$, $pr_1|_{N_8(e,1)}$ is not injective.
\end{example}
The Example \ref{e1} shows that a digitally continuous surjection with path lifting property is not necessarily a digital covering map. By the following example, we show that a digitally continuous surjection with uniqueness of digital path lifts property is not necessarily a digital covering map.
\begin{example}
Consider the map $h:\Z^+\lo SC_8^{2,4}=:(c_i)_{i\in [0,3]_{\Z}}$ given by $h(i)=c_{i\ mod\ 4}$, where $\Z^+:=\{k\in\Z| k\geq 0\}$.
let $\gamma:[0,1]_{\Z}\lo SC_8^{2,4}$ defined by $\gamma(0)=c_0$ and $\gamma(1)=c_3$. Since $0\in h^{-1}(c_0)$ and there is no lifting of $\gamma$ beginning at $0$, $h$ has not digital path lifting property.

Now, Let $\al,\beta:[0,m]_{\Z}\lo \Z^+$ be two $2$-paths in which $h\circ\al=h\circ\beta$ and $\al(0)=\beta(0)=d$.
We show that $\al=\bt$. By contrary, suppose that there is $s\in [0,m]_{\Z}$ such that $\al(s)\neq \bt(s)$. We may assume that $s$ is the smallest $t\in [0, m]_{\Z}$ such that $\al(t)\neq\bt(t)$. Thus we have the following:
$$\begin{cases}
\al(s)\neq\bt(s),\\
\al(t)=\bt(t),\  \textit{for \ all\ } t\in [0, s-1]_{\Z},\\
h\circ\al(t)=h\circ\bt(t),\  \textit{for \ all\ } t\in [0, m]_{\Z}.
\end{cases}
$$
If $k:=\al(s-1)=\bt(s-1)$, then we have
$$\begin{cases}
\al(s)=k+1,\\
\bt(s)=k-1.
\end{cases}
or\  \begin{cases}
\al(s)=k-1,\\
\bt(s)=k+1.
\end{cases}
$$
But $h\circ\al(s)=h\circ\bt(s)$ follows that $h(k-1)=h(k+1)$ which is contradiction because $h(j)=h(k)$ if an only if $j=k\ mod\ 4$.
\end{example}
%\\\\\\\\\\\\\\\\\\\\\\\\\\\\\\\\\\\\\\\\\\\\\\\\\\\\\\\\\\\\\\\\\\\\\\\\\\\
%=============================================================================
%////////////////////////////////////////////////////////////////////////////////
In the following proposition we give some basic properties of maps with u.p.l which are essential in our results and make the proofs more shorter and simpler.
\begin{proposition}\label{p1}
Let $\peb$ be a $\kl$-continuous surjection map with u.p.l. Then
\begin{itemize}
\item[(i)] If $e\stackrel{\ka}{\leftrightarrow}e'$ then $p(e)\neq p(e')$.
\item[(ii)] If $e\stackrel{\ka}{\leftrightarrow}e'$, $e\stackrel{\ka}{\leftrightarrow}e''$ and $e'\neq e''$ then $p(e'
)\neq p(e'')$.
\item[(iii)] If $p(e)\stl p(e')$ then there is a unique element $e''\in p^{-1}(p(e'))$ such that $e\stk e''$.
\item[(iv)] If $p(e)\stl b$ then there is a unique $e'\in p^{-1}(b)$ such that $e\stk e'$.
\item[(v)] If $b\stl b'$ then for every $e\in p^{-1}(b)$ there is a unique element $e'\in p^{-1}(b')$ such that $e\stk e'$.
\end{itemize}
\end{proposition}
\begin{proof}
(i) Let $\al,\bt:[0,1]_{\Z}\lo E$ be defined by $\al(0)=\bt(0)=e$, $\al(1)=e$ and $\bt(1)=e'$ which are $\ka$-pathes by assumption. If $p(e)= p(e')$ then $p\circ\al=p\circ\bt$ while $\al\neq \bt$. This is contradiction and hence $p(e)\neq p(e')$.\\
(ii) Let $\al,\bt:[0,1]_{\Z}\lo E$ be defined by $\al(0)=\bt(0)=e$, $\al(1)=e'$ and $\bt(1)=e''$ which are $\ka$-pathes by assumption. If $p(e')= p(e'')$ then $p\circ\al=p\circ\bt$ while $\al\neq \bt$. Hence $p(e')\neq p(e'')$.\\
(iii) Define $\al:[0,1]_{\Z}\lo B$ by $\al(0)=p(e)$ and $\al(1)=p(e')$ which is a $\la$-path. By path lifting property, there is a lifting $\wt{\al}:[0,1]_{\Z}\lo E$ beginning at $e$. Since $p\circ\wt{\al}=\al$, $\wt{\al}(1)\in p^{-1}(\al(1))=p^{-1}(p(e'))$. Now it is sufficient to let $e''=\wt{\al}(1)$ because $\wt{\al}(0)\stk\wt{\al}(1)$. Uniqueness comes from part (ii).\\
(iv) The proof is similar to the proof of (iii).\\
(v) This is also similar to (iii) because maps with path lifting property are surjective.
\end{proof}
\begin{definition}
Let $\peb$ be a $\kl$-continuous map and $e\in E$. We say that $e$ is a conciliator point for $p$ if there exist $e',e''\in N_{\ka}(e,1)$ for which $e'\stackrel{\ka}{\nleftrightarrow}e''$ and $p(e')\stackrel{\la}{\leftrightarrow}p(e'')$.
\end{definition}
In General Topology, every covering map has u.p.l but every map with u.p.l is not generally a covering map \cite{Sp}. In fact, the domain and codomain in the maps with u.p.l that are not a covering map have locally complicated behaviors and this will ruin it to be a covering map. We show that u.p.l is enough for a map without conciliator point to be a digital covering map.
\begin{theorem}\label{t1}
A $\kl$-continuous surjection map $\peb$ is a digital $\kl$-covering if it has u.p.l and has no conciliator point.
\end{theorem}
\begin{proof}
Let $b\in B$ and $e\in p^{-1}(N_{\la}(b,1))$. We show that $e\in\D\bigsqcup_{j\in J} N_{\ka}(e_j,1)$, where $p^{-1}(b)=\{e_j\}_{j\in J}$. If $e\in p^{-1}(b)$, then $e$ is one of the $e_j$'s and assertion is obvious.
If $e\notin p^{-1}(b)$, then $p(e)\stl b$ and by Proposition \ref{p1}, there is $e_j\in p^{-1}(b)$ such that $e\stk e_j$ which implies $e\in N_{\ka}(e_j,1)$, as desired. If $e\in\D\bigsqcup_{j\in J} N_{\ka}(e_j,1)$, then there is $j_0\in J $ such that $e\in N_{\ka}(e_{j_0},1)$ and hence either $e=e_{j_0}$ or $e\stk e_{j_0}$ which implies that $p(e)=b$ or $p(e)\stl b$. This means $p(e)\in N_{\la}(b,1)$ and therefore $e\in p^{-1}(N_{\la}(b,1))$.

Let $x\in N_{\ka}(e_i,1)\cap N_{\ka}(e_j,1)$, for $i\neq j$. Then $x\stk e_i$ and $x\stk e_j$ and by Proposition \ref{p1}, $b=p(e_i)\neq p(e_j)=b$ which is contradiction.

For every $j\in J$, the restriction map $p|_{N_{\ka}(e_j,1)}:N_{\ka}(e_j,1)\rightarrow N_{\la}(b,1)$ is injective by Proposition \ref{p1}, part i and ii and also is surjective by part v. For continuity of $(p|_{N_{\ka}(e_j,1)})^{-1}$, let $b',b''\in N_{\la}(b,1)$ be two $\la$-adjacent points. Since $p|_{N_{\ka}(e_j,1)}$ is bijective, there are $e',e''\in N_{\ka}(e_j,1)$ such thet $p(e')=b'$ and $p(e'')=b''$. If $e'\stackrel{\ka}{\nleftrightarrow}e''$ then $e_j$ is conciliator point of $p$ which is contradiction. Hence $e'\stackrel{\ka}{\leftrightarrow}e''$ and so $(p|_{N_{\ka}(e_j,1)})^{-1}$ is continuous.

\end{proof}

In the following, we give an example of a continuous surjection map with u.p.l  in which is not a digital covering map. This shows that the absence of conciliator points is essential.
\begin{example}
Let $E=\mathbb{Z}$, $B=\{b_0=(0,0), b_1=(1,0), b_2=(0,1)\}$ and $p:E\longrightarrow B$ be defined by $p(i)=b_{i~mod~3}$.
Then $p$ is a $(2, 8)$-continuous surjection that has the unique path lifting property and Also $p$ has some conciliator point, for example $0, 3, 6, . . .$. But $p$
is not a digital $(2, 8)$-covering because $p$ does not satisfy condition (3) of Definition \ref{d1}: For example,
 $N_{2}(0,1) = \{-1, 0, 1\}$ and so the inverse of the restriction of $p$ to $N_{2}(0,1)$ is
not $(8, 2)$-continuous, because it maps two $8$-adjacent points $b_1$ and $b_2$ in $N_{8}(b_0,1) = B$ to two distinct points of
$N_{2}(0,1) = \{–1, 0, 1\}$ that are not $2$-adjacent.
\end{example}

%\\\\\\\\\\\\\\\\\\\\\\\\\\\\\\\\\\\\\\\\\\\\\\\\\\\\\\\\\\\\\\\\\\\\\\\\\\\\\\\
%==================================================================================
%/////////////////////////////////////////////////////////////////////////////////
\section{Coverings are derived from local isomorphisms }
Like what happened to covering maps and u.p.l in General Topology, every covering map is a local isomorphism, but every local isomorphism is not necessarily a covering map. Obviously and by definitions, every digital covering map is a local isomorphism. Han \cite{H3}  gave an example showing that a local isomorphism is not a covering map. We will find a gap in his example and will show that in Digital Topology, local isomorphisms are digital covering maps.
\begin{example}\cite{H3}
Assume $$X=\{q_0=(x_1,y_1),q_1=(x_1-1,y_1+1),q_2=(x_1-2,y_1+1),q_3=(x_1-3,y_1),$$
$$q_4=(x_1-3,y_1-1),q_5=(x_1-2,y_1-2),q_6=(x_1-1,y_1-1)\}\sub (\Z^2,8)$$ and
$$Y=\{v_0=(a,b),v_1=(a-1,b+1),v_2=(a-2,b),v_3=(a-1,b-1)\}\sub (\Z^2,8).$$
 Han claimed that the map $h:X\lo Y$ with $h(q_i)=v_{i(mod\ 4)}$ is a local $8$-isomorphism and is not a $(8,8)$-covering map because the third assumption in the definition of covering map is not satisfied in point $v_0$. Although his assertion about the point $v_0$ is true, but $h$ is not $8$-continuous and hence is not a local isomorphism. In fact $q_0$ and $q_6$ are $8$-adjacent, while $h(q_0)$ and $h(q_6)$ are not same or $8$-adjacent.
\end{example}
\begin{theorem}\label{t2}
A $\kl$-continuous surjection map $\peb$ is a digital $\kl$-covering map if and only if it is a local isomorphism.
\end{theorem}
\begin{proof}
By Theorem \ref{t1} it suffices to show that every local isomorphism has u.p.l and has no conciliator point.
Let $\al:[0,m]_{\Z}\lo B$ be a digital $\la$-path with $b_0=\al(0)$. Since $p$ is surjective, there exists $e_0\in p^{-1}(b_0)$ and by assumption, $p|_{N_{\ka}(e_0,1)}:N_{\ka}(e_0,1)\rightarrow N_{\la}(b_0,1)$ is a $\kl$-isomorphism. Since $\al(0)\stl \al(1)$, there exists $e_1\in p^{-1}(\al(1))$ such that $e_1\in N_{\ka}(e_0,1)$. If $\al(0)= \al(1)$, put $e_1=e_0$. Inductively, there is $e_i\in p^{-1}(\al(i))$ such that $e_i\in N_{\ka}(e_{i-1},1)$, for any $0< i\leq m$. Now, define $\wt{\al}:[0,m]_{\Z}\lo E$ by $\wt{\al}(i)=e_i$ which is a $\ka$-path because $\wt{\al}(i)=e_i\stk e_{i-1}=\wt{\al}(i-1)$. Since $e_i\in p^{-1}(\al(i))$, $\wt{\al}$ is a lifting of $\al$.

For uniqueness of path lifts property, consider two paths
$\al,\beta:[0,m]_{\Z}\lo E$ in which $p\circ\al=p\circ\beta$ and $\al(0)=\beta(0)$. By contradiction, assume that $\al\neq\bt$. Since $\al(0)=\beta(0)$, the set $\{i\in[0,m]; \al(i)\neq\bt(i)\}$ has minimum $l$. Hence $\al(l)\neq \bt(l)$, while $\al(l-1)=\bt(l-1)$. But the map
$$p|_{N_{\ka}(\al(l-1),1)}:N_{\ka}(\al(l-1),1)\rightarrow N_{\la}(p\circ\al(l-1),1)$$
 is an isomorphism which implies that $p\circ\al(l)\neq p\circ\bt(l)$, for $\al(l),\bt(l)\in N_{\ka}(\al(l-1),1) $. This contradicts $p\circ\al=p\circ\beta$ and therefore $\al=\bt$.\\
 Now, let $e\in E$ be a conciliator point of $p$. Then, there are $e',e''\in N_{\ka}(e,1)$ such that $e'\stackrel{\ka}{\nleftrightarrow}e''$ and $p(e')\stackrel{\la}{\leftrightarrow}p(e'')$ which implies that the restricted map $p|_{N_{\ka}(e,1)}:N_{\ka}(e,1)\longrightarrow N_{\la}(p(e),1)$ is not an isomorphism. This is a contradition and hence $p$ has no conciliator point.
\end{proof}

If in the definition of digital covering map, we replace the $N_{\la}(b, 1)$ by $N_{\la}(b, n)$, for $n\in \N$, the map is called \textbf{radius $n$ covering map} and hence it is radius $n$ local isomorphism. Radius $n$ coverings, particulary radius 2 coverings are very important in the digital covering theory because some essential theorems in Algebraic Topology are satisfied in Digital Topology if covering maps will be radius 2 covering maps.
\begin{theorem}\cite{B3,B4,H2,H4}\label{T3}
Let $p : (E,\ka)\lo (B,\la)$ be a $(\ka,\la)$-covering map such that $p(e_0) = b_0$. Suppose
that $p$ is a radius 2 local isomorphism. Then\\
(1) For $\ka$-paths $\al,\bt : [0,m]_{\Z} \lo E$ starting at $e_0$,
if there is a $\la$-homotopy in $B$ from $p\circ\al$ to $p\circ\bt$ that holds the endpoints fixed, then
$\al(m)=\bt(m)$, and there is a $\ka$-homotopy in $E$ from $\al$ to $\bt$ that holds the endpoints
fixed.\\
(2) The induced homomorphism $p_*:\pi_1^{\ka}(E,e_0)\lo \pi_1^{\la}(B,b_0)$ is a monomorphism.\\
(3) For a given $\ka'$-connected space $X$ with $x_0\in X$, any $(\ka',\la)$-continuous map $\vf:(X,\ka')\lo (B,\la)$ with $\vf(x_0)=b_0$ has a digital lifting $\wt{\vf}:(X,\ka')\lo (E,\ka)$ for which $\wt{\vf}(x_0)=e_0$ if and only if $\vf_*\big(\pi_1^{\ka'}(X,x_0)\big)\sub p_*\big(\pi_1^{\ka}(E,e_0)\big)$.
\end{theorem}
For more results in digital covering maps based on 2-radius  property, see \cite{B3,B4,H2}. In the following, we give a loop criterion for a digital covering to be a radius $n$ covering map. But we need this lemma.
\begin{lemma}\label{L1}
Let $p : (E,\ka)\lo (B,\la)$ be a $(\ka,\la)$-covering map, $e\in E$, $e'\in N_{\ka}(e,1)$ and $e'\neq e''\in N_{\ka}(e,2)$. Then $p(e')\neq p(e'')$.
\end{lemma}
\begin{proof}Let $b=p(e)$ and by contrary assume that $b':=p(e')= p(e'')$.
 Since $p|_{N_{\ka}(e,1)}$ is an isomorphism, $e''\in {N_{\ka}(e,2)}-{N_{\ka}(e,1)}$
and hence there exists $f\in N_{\ka}(e,1)$ such that $e''\stk f$. Since $p(e')=p(e'')\stk p(f)$ and $p|_{N_{\ka}(e,1)}$ is an isomorphism and also $p(e'), p(f)\in N_{\la}(b,1)$, we have $f\stk e'$. Now, $p|_{N_{\ka}(f,1)}$ is an isomorphism, $e',e''\in N_{\ka}(f,1)$ and $p(e')=p(e'')$ which is a contradiction.
\end{proof}

\begin{theorem}\label{T4}
Let $p : (E,\ka)\lo (B,\la)$ be a $(\ka,\la)$-covering map. $p$ is a radius $n$ covering map if and only if every lifting of any simple $\la$-loop with length at most $2n+1$ is a simple $\ka$-loop.
\end{theorem}
\begin{proof}
Let $p$ be a radius $n$ covering map and in the worst conditions, $\al:[0,2n+1]_{\Z}\lo B$ be a simple $\la$-loop with length $2n+1$. Let $b:=\al(0)$ and assume that $e\in p^{-1}(b)$. Consider two $\la$-paths $\al_1,\al_2:[0,n]_{\Z}\lo B$ defined by $\al_1(k)=\al(k)$ and $\al_2(k)=\al(2n+1-k)$. In fact, $\al_1$ is $\al|_{[0,n]}$ and $\al_2$ is $(\al|_{[n+1,2n]})^{-1}$. Since $p|_{N_{\ka}(e,n)}:N_{\ka}(e,n)\lo N_{\la}(b,n)$ is a $(\ka,\la)$-isomorphism, $\wt{\al_1}:=(p|_{N_{\ka}(e,n)})^{-1}\circ\al_1$ and $\wt{\al_2}:=(p|_{N_{\ka}(e,n)})^{-1}\circ\al_2$ are liftings of $\al_1$ and $\al_2$, respectively and $\wt{\al_1}(0)=\wt{\al_2}(0)=e$. Also, $\wt{\al_1}(n)\stackrel{\ka}{\leftrightarrow}\wt{\al_2}(n)$ because ${\al_1}(n)\stackrel{\ka}{\leftrightarrow}{\al_2}(n)$ and $p|_{N_{\ka}(b,n)}$ is an isomorphism. Define $\mu:[0,2n+1]\longrightarrow E$ by
$$\mu(i)=\begin{cases}
\wt{\al_1}(i) ~~~\qquad~~\qquad~~~~ 0\leq i\leq n,\\
\wt{\al_2}(i) ~~~~\qquad~~~\qquad~~ n+1\leq 2n+1.
\end{cases}$$

Obviously, $\mu$ is a simple $\ka$-loop and $p\circ\mu=\al$, as desired.

For the converse, we use induction to show that $p|_{N_{\ka}(e,n)}$ is an isomorphism, for every $b\in B$ and any $e\in p^{-1}(b)$.

 Assume that all liftings of every simple $\la$-loop with length $5$ is a simple $\ka$-loop. We must prove that $p|_{N_{\ka}(e,2)}$ is an isomorphism, for every $b\in B$ and any $e\in p^{-1}(b)$. Since $p|_{N_{\ka}(e,1)}$ is an isomorphism, if $x\in N_{\la}(b,1)$, Then there is $y\in N_{\ka}(e,1)$ such that $p(y)=x$. Assume $x\in N_{\la}(b,2)-N_{\la}(b,1)$. Let $\al:[0,2]_{\Z}\lo N_{\la}(b,2)$ be a $\la$-path from $b$ to $x$. Then there is a unique lifting $\wt{\al}:[0,2]_{\Z}\lo N_{\ka}(e,2)$ beginning at $e$ such that $y:=\wt{\al}(2)\in N_{\ka}(e,2)$ and $p(y)=x$. Hence $p|_{N_{\ka}(e,2)}$ is onto.

 For injectivity, by contrary assume that there are $y,y'\in N_{\ka}(e,2)$ such that $x:=p(y)=p(y')$. Since $p|_{ N_{\ka}(e,1)}$ is an isomorphism, $y,y'\notin  N_{\ka}(e,1)$. If $y\in  N_{\ka}(e,1)$ and $y'\in  N_{\ka}(e,2)$, then by Lemma \ref{L1}, $p(y)\neq p(y')$.
 Hence we can consider $y,y'\in N_{\ka}(e,2)-N_{\ka}(e,1)$. By part i of Proposition \ref{p1}, $y,y'$ are not $\ka$-adjacent.
  %Since $p|_{N_{\ka}(e,1)}$ is an isomorphism and $p$ do not take the same value in two $\ka$-adjacent point (Proposition \ref{p1}, i), $x\in N_{\la}(b,2)-N_{\la}(b,1)$.
  There are two points $e_1,e_1'\in N_{\ka}(e,1)$ such that $e_1\stk y$ and $e_1'\stk y'$. If $e_1=e_1'$, then we have two liftings beginning at $e_1$ for the path $\delta:[0,1]_{\Z}\lo B$, by $\delta(0)=p(e_1)$ and $\delta(1)=x$ which is contradiction. Hence $e_1\neq e_1'$.

  If $p(e_1)\stl p(e_1')$, then $e_1\stk e_1'$ (because $p|_{N_{\ka}(e,1)}$ is an isomorphism) and since $y\in N_{\ka}(e_1,1)$ and $y'\in N_{\ka}(e_1,2)$, by Lemma \ref{L1} we have $p(y)\neq p(y')$. Therefore $p(e_1)\stackrel{\la}{\nleftrightarrow}p(e_1')$. Let $b_1=p(e_1)$ and $b_1'=p(e_1')$. Define
  $$\begin{cases}
\al:[0,4]_{\Z}\lo B,\\
\al(0)=\al(4)=x,\\
\al(1)=b_1,\al(2)=b,\al(3)=b_1',\end{cases}$$
which is a simple $\la$-loop based at $b$ with length $4$ and hence all of its liftings are closed. But $\gamma:[0,4]_{\Z}\lo E$ defined by $\gamma(0)=y,~\gamma(1)=e_1,~\gamma(2)=e,~\gamma(3)=e_1'$ and $\gamma(4)=y'$ is a lifting of $\al$ which is not closed. Hence $y=y'$ and therefore $p|_{N_{\ka}(e,2)}$ is injective.

 Obviously $p|_{N_{\ka}(e,2)}$ is continuous. For continuity of $\big(p|_{N_{\ka}(e,2)}\big)^{-1}$, let $b',b''\in N_{\la}(b,2)$ such that $b'\stackrel{\la}{\leftrightarrow}b''$. Then\\
 \begin{itemize}
 \item If $b',b''\in N_{\la}(b,1)$, then $\big(p|_{N_{\ka}(e,2)}\big)^{-1}(b')\stackrel{\ka}{\leftrightarrow}\big(p|_{N_{\ka}(e,2)}\big)^{-1}(b'')$ because $p|_{N_{\ka}(e,1)}$ is an isomorphism.
 \item If $b',b''\in N_{\la}(b,2)- N_{\la}(b,1)$, then we can define $\la$-simple loop $\al:[0,5]_{\mathbb{Z}}\longrightarrow N_{\la}(b,2)$
 by $\al(0)=b, \al(1)=b_1,\al(2)=b',\al(3)=b'',\al(4)=b_1'$ and $\al(5)=b$, where $b_1,b_1'\in N_{\la}(b,1)$. By assumption, all of its liftings are closed and by u.p.l, there exists unique $\ka$-simple loop $\wt{\al}$ started at $e$ such that $p\circ\wt{\al}=\al$. Since $p|_{N_{\ka}(e,2)}$ is bijective, $\big(p|_{N_{\ka}(e,2)}\big)^{-1}(b')=\wt{\al}(2)\stackrel{\ka}{\leftrightarrow}\wt{\al}(3)=\big(p|_{N_{\ka}(e,2)}\big)^{-1}(b'')$, as desired.
 \item If $b'\in N_{\la}(b,2)$ and $b''\in N_{\la}(b,1)$, by a similar way as in the previous item, we can define a $\la$-simple loop with length 4 in $N_{\la}(b,2)$ and deduce that $\big(p|_{N_{\ka}(e,2)}\big)^{-1}(b')\stackrel{\ka}{\leftrightarrow}\big(p|_{N_{\ka}(e,2)}\big)^{-1}(b'')$.
  \end{itemize}

 Therefore $p|_{N_{\ka}(e,2)}$ is an isomorphism.

If $p|_{N_{\ka}(e,n-1)}$ is an isomorphism and all liftings of every simple $\la$-loop with length $2n+1$ is a simple $\ka$-loop, a similar method shows that $p|_{N_{\ka}(e,n)}$ is an isomorphism.
\end{proof}
In the following example we show that we can not replace 2n+1 by
2n in the Theorem \ref{T4}.
\begin{example}
Let $\al:[0,5]_{\mathbb{Z}}\longrightarrow \mathbb{Z}^3$ be a 26-simple loop and denote $\al(i)$ by $b_i$. Define $p:\mathbb{Z}\longrightarrow B=\{b_0,b_1,b_2,b_3,b_4\}$ by $p(i)=b_{i~mod~5}$. Readily, $p$ is a (2,26)-covering map. As $B$ is a simple
loop and $|B| > 2n = 4$, there is no non-trivial simple 26-loop in $B$ of length 4 or less and so the hypothesis that "every lifting of any simple 26-loop with length at most $2n$ is a simple 2-loop" is
satisfied. Also, there exist a 26-loop with length 5 such that has no closed lifting, namely $\al$. But $p$ is not a radius 2 covering map as it is not a radius 2 local isomorphism. For example, $N_2(0, 2) = [-2, 2]_{\mathbb{Z}}$ and so $\big(p|_{N_2(0, 2)}\big)^{-1}$ maps the 26-adjacent points $b_2$ and $b_3$ in $N_{26}(b_0,2) = B$ to two distinct points of $N_2(0, 2) = [-2, 2]_{\mathbb{Z}}$ that
 are not 2-adjacent (namely the points 2 and $-2$).
\end{example}

Although in General Topology, the induced homomorphism on fundamental groups by a covering map is a monomorphism, but this is not true in digital topology \cite{B3}. We have the following corollary by using Theorem 4.3 and Theorem 4.5.

\begin{corollary}
Let $p : (E,\ka)\lo (B,\la)$ be a $(\ka,\la)$-covering map. The induced homomorphism $p_*:\pi_1^{\ka}(E,e_0)\lo \pi_1^{\la}(B,b_0)$ is a monomorphism if every lifting of any simple $\la$-loop with length $5$ is closed.
\end{corollary}
\begin{proof}
Since every lifting of any simple $\la$-loop with length $5$ is closed, $p$ is a radius 2 covering map by Theorem \ref{T4} and so by part 2
 of Theorem \ref{T3}, $p$ is a monomorphism
\end{proof}
Using Theorem 4.5, we can restate Theorem 4.3:
\begin{corollary}(Digital Lifting Criteria)\label{C1}

Let $p : (E,\ka)\lo (B,\la)$ be a continuous surjection map, $X$ be a $\ka'$-connected space with $x_0\in X$ and $\vf:(X,\ka')\lo (B,\la)$ be a $(\ka',\la)$-continuous map with $\vf(x_0)=b_0$. Then the existence
of a digital lifting $\wt{\vf}:(X,\ka')\lo (E,\ka)$ for which $\wt{\vf}(x_0)=e_0$ is equivalent to the algebraic assumption $\vf_*\big(\pi_1^{\ka'}(X,x_0)\big)\sub p_*\big(\pi_1^{\ka}(E,e_0)\big)$ if at least one of the following conditions hold:\\
(a) $p$ has u.p.l, has no conciliator point and every lifting of any simple $\la$-loop with length $5$ is closed.\\
(b) $p$ is a radius 2 local isomorphism.
\end{corollary}
\begin{proof}
(a) If $p$ has u.p.l and has no conciliator point, then it is a covering map, by Theorem \ref{t1} and since every lifting of any simple $\la$-loop with length $5$ is closed, $p$ is a radius 2 covering map. Part 3 of Theorem \ref{T3} implies the existence
of desired $\wt{\vf}:(X,\ka')\lo (E,\ka)$. \\
(b) If $p$ is a radius 2 local isomorphism, Theorem \ref{t2} implies that it is a radius 2 covering map and so the existence
of desired $\wt{\vf}:(X,\ka')\lo (E,\ka)$ comes from part a.
\end{proof}
We know that every digital radius 2 covering map of a simply connected digital space is an isomorphism (\cite[Corollary 3.14]{B4}). Theorem \ref{t2} implies that continuous surjection radius $2$ local isomorphisms are radius $2$ covering maps and so we have the following corollary.
\begin{corollary} Let  continuous surjection map $p : (E,\ka)\lo (B,\la)$ be a radius $2$ local isomorphism. Then it is a $(\ka,\la)$-isomorphism if $B$ is simply connected.
\end{corollary}

\end{document}